\documentclass[10pt, reqno]{amsart}

\usepackage[centertags]{amsmath}
\usepackage{amsfonts}
\usepackage{amssymb}
\usepackage{amsthm}
\usepackage{newlfont}
\usepackage{fancyhdr}

\def\BibTeX{{\rm B\kern-.05em{\sc i\kern-.025em b}\kern-.08em
    T\kern-.1667em\lower.7ex\hbox{E}\kern-.125emX}}

\newcommand{\EE}{\mathbb{E}}

\newcommand{\E}{\mathrm{E}}
\newcommand{\PP}{\mathrm{P}}
    \newcommand{\dto}{\xrightarrow{d}}
    
    \newcommand{\vto}{\xrightarrow{v}}
    \newcommand{\fidi}{\xrightarrow{\text{fidi}}}

\newcommand{\eqd}{\stackrel{d}{=}}
 \newcommand{\floor}[1]{\lfloor#1\rfloor}
\theoremstyle{plain}
\newtheorem{theorem}{Theorem}[section]

\newtheorem{lemma}[theorem]{Lemma}
\newtheorem{proposition}[theorem]{Proposition}

\theoremstyle{definition}
\newtheorem{remark}[theorem]{Remark}
\newtheorem{example}[theorem]{Example}
\newtheorem{cond}[theorem]{Condition}

\numberwithin{equation}{section}

\begin{document}

\title[Functional weak convergence of partial maxima processes] 
{Functional weak convergence of partial maxima processes}

%
\author{Danijel Krizmani\'{c}}

\address{Danijel Krizmani\'{c}\\ Department of Mathematics\\
        University of Rijeka\\
        Radmile Matej\v{c}i\'{c} 2, 51000 Rijeka\\
        Croatia}
\email{dkrizmanic@math.uniri.hr}




\begin{abstract}
For a strictly stationary sequence of nonnegative regularly varying random variables $(X_{n})$ we study functional weak convergence of partial maxima processes $M_{n}(t) = \bigvee_{i=1}^{\lfloor nt \rfloor}X_{i},\,t \in [0,1]$ in the space $D[0,1]$ with the Skorohod $J_{1}$ topology. Under the strong mixing condition, we give sufficient conditions for such convergence when clustering of large values do not occur. We apply this result to stochastic volatility processes. Further we give conditions under which the regular variation property is a necessary condition for $J_{1}$ and $M_{1}$ functional convergences in the case of weak dependence. We also prove that strong mixing implies the so-called Condition $\mathcal{A}(a_{n})$ with the time component.
\end{abstract}

\maketitle

\section{Introduction}
\label{intro}

Let $(X_{n})$ be a strictly stationary sequence of nonnegative random variables. Denote by $M_{n} = \max \{X_{i} : i=1,\ldots,n \}$, $n \geq 1$, and let $(a_{n})$ be a sequence of positive real numbers such that
\begin{equation}\label{eq:an}
 n \PP(X_{1}>a_{n}) \to 1 \qquad \textrm{as} \ n \to \infty.
 \end{equation}
If the sequence $(X_{n})$ is i.i.d.~then it is well known (see for example Resnick~\cite{Re07}, Propostion 7.1) that
\begin{equation}\label{e:FLTe}
\frac{M_{n}}{a_{n}} \dto S,
 \end{equation}
 for some non-degenerate random variable $S$ if and only if $X_{1}$ is \emph{regularly varying with index} $\alpha >0$, that is,
\begin{equation}\label{e:regvarosn2}
\PP (X_{1} > x) = x^{-\alpha} L(x),
\end{equation}
where $L(\,\cdot\,)$ is a slowly varying function at $\infty$, i.e. for every $t>0$, $L(tx)/L(x) \to 1$ as $x \to \infty$. In this case $S$ is a Fr\'{e}chet random variable with distribution
 $$ \PP(S \leq x)  = e^{-x^{-\alpha}}, \qquad x > 0.$$
The regular variation property (\ref{e:regvarosn2}) is equivalent to
\begin{equation}\label{eq:regvar}
 n \PP \Big( \frac{X_{1}}
{a_{n}} \in \cdot \Big) \vto \mu (\,\cdot\,), \qquad n \ \to \infty,
\end{equation}
with $\mu$ being a measure of the form
 $ \mu(dx) = \alpha x^{-\alpha -1}1_{(0,\infty)}(x)\,dx.$

 The functional generalization of (\ref{e:FLTe}) has been studied extensively in probability literature. Define the partial maxima processes
$$ M_{n}(t) = \bigvee_{i=1}^{\lfloor nt
   \rfloor}\frac{X_{i}}{a_{n}}, \qquad t \in [0,1].$$
Here $\floor{x}$ represents the integer part of the real number $x$. In functional limit theory one investigates the asymptotic behavior of the processes $M_{n}(\,\cdot\,)$ as $n \to \infty$. Since the sample paths of $M_{n}(\,\cdot\,)$ are elements of the space $D[0,1]$ of all right-continuous real valued functions on $[0,1]$ with left limits, weak convergence of distributions of $M_{n}(\,\cdot\,)$ are is considered with respect to the one of the Skorohod topologies on $D[0,1]$ introduced in Skorohod~\cite{Sk56}.

In the i.i.d.~case Lamperti~\cite{La64} (see also Proposition 7.2 in Resnick~\cite{Re07}) showed that weak convergence of processes $M_{n}(\,\cdot\,)$ in $D[0,1]$ with the Skorohod $J_{1}$ topology is equivalent to the regular variation property of $X_{1}$, with an extremal process as a limit. In the dependent case, Adler~\cite{Ad78} obtained $J_{1}$ functional convergence with the weak dependence conditions similar to conditions $D$ and $D'$ introduced by Leadbetter~\cite{Le74} and~\cite{Le76}. Mori~\cite{Mo77} and Durrett and Resnick~\cite{DuRe78} obtained $J_{1}$ convergence of the maxima processes from the convergence of a certain two-dimensional point processes. The $J_{1}$ topology is appropriate when large values of $X_{n}$ do not cluster. A standard tool in describing clustering of large values is the extremal index of the sequence $(X_{n})$, which is equal to $1$ when large values do not cluster and less than $1$ when clustering occurs. In the latter case $J_{1}$ convergence in general fails to hold, although convergence with respect to the weaker Skorohod $M_{1}$ topology might still hold. Recently Krizmani\'{c}~\cite{Kr14} obtained $M_{1}$ functional convergence under the properties of weak dependence and joint regular variation for the sequence $(X_{n})$.

Since we study
extremes of random processes, nonnegativity of random variables
$X_{n}$ is not a restrictive assumption. First, we introduce the essential
ingredients about point processes, regular variation and weak dependence in Section
\ref{s:one}. Here we also give a formal proof that strong mixing implies the so-called Condition $\mathcal{A}(a_{n})$ of Davis and Hsing~\cite{DaHs95} with the time component. In Section~\ref{s:two}, for a strictly stationary sequence of nonnegative regularly varying random variables with extremal index equal to $1$ we show $J_{1}$ convergence of the partial maxima process $M_{n}(\,\cdot\,)$ under the strong mixing condition. The regular variation property is a necessary condition for the $J_{1}$ convergence of the partial maxima process in the i.i.d.~case (c.f. Proposition 7.2 in Resnick~\cite{Re07}). In Section~\ref{s:three} we extend this result to the weak dependent case when clustering of large values do not occur. We further show the necessity of regular variation also when we have convergence in the weaker Skorohod $M_{1}$ topology when clustering of large values occurs. Some ideas and techniques used in this paper already appeared in Krizmani\'{c}~\cite{Kr14b} where functional weak convergence for partial sum processes was investigated.

\section{Preliminaries}\label{s:one}

Let $\mathbb{E} = (0, \infty]$. The space $\mathbb{E}$ is equipped with the topology by which a set $B \subseteq \EE$
has compact closure if and only if there exists $u > 0$ such that $B \subseteq \EE_u = (u, \infty]$. It suffices to take the following metric
 $$ \rho (x,y) = \bigg| \frac{1}{x} - \frac{1}{y} \bigg|, \qquad x, y \in \EE.$$
 Let $\mathbf{M}_{+}(\mathbb{E})$ be the class of all Radon measures on $\mathbb{E}$. A useful topology for $\mathbf{M}_{+}(\mathbb{E})$ is the vague topology which renders $\mathbf{M}_{+}(\mathbb{E})$ a complete separable metric space. If $\mu_{n} \in \mathbf{M}_{+}(\mathbb{E})$, $n \geq 0$, then $\mu_{n}$ converges vaguely to $\mu_{0}$ (written $\mu_{n} \xrightarrow{v} \mu_{0}$) if
 $\int f\,d \mu_{n} \to \int f\,d \mu_{0}$ for all $f \in C_{K}^{+}(\mathbb{E})$, where $C_{K}^{+}(\mathbb{E})$ denotes the class of all nonnegative continuous real functions on $\mathbb{E}$ with compact support.

A Radon point measure is an element of $\mathbf{M}_{+}(\mathbb{E})$ of the form $m = \sum_{i}\delta_{x_{i}}$, where $\delta_{x}$ is the Dirac measure. By $\mathbf{M}_{p}(\mathbb{E})$ we denote the class of all Radon point measures. A point process on $\mathbb{E}$ is an $\mathbf{M}_{p}(\mathbb{E})$--valued random element, defined on a given probability space. For more background on the theory of point processes we refer to Kallenberg~\cite{Ka83}.

We say that a strictly stationary $\mathbb{R}_{+}$--valued process $(\xi_{n})$ is \emph{jointly regularly varying} with index
$\alpha \in (0,\infty)$ if for any nonnegative integer $k$ the
$k$-dimensional random vector $\boldsymbol{\xi} = (\xi_{1}, \ldots ,
\xi_{k})$ is multivariate regularly varying with index $\alpha$ (see Basrak et al.~\cite{BKS}).
Theorem~2.1 in Basrak and Segers~\cite{BaSe} provides a convenient
characterization of joint regular variation: it is necessary and
sufficient that there exists a process $(Y_n)_{n \in \mathbb{Z}}$
with $\PP(Y_0 > y) = y^{-\alpha}$ for $y \geq 1$ such that as
$x \to \infty$,
\begin{equation}\label{e:tailprocess}
  \bigl( (x^{-1}\xi_n)_{n \in \mathbb{Z}} \, \big| \, \xi_0 > x \bigr)
  \fidi (Y_n)_{n \in \mathbb{Z}},
\end{equation}
where "$\fidi$" denotes convergence of finite-dimensional
distributions. The process $(Y_{n})$ is called
the \emph{tail process} of $(\xi_{n})$.

The stochastic processes that we consider have discontinuities, and therefore for the function space of its sample paths we take the space $D[0,1]$ of real valued c\`{a}dl\`{a}g functions on $[0,1]$. Usually the space $D[0,1]$ is endowed with the Skorohod $J_{1}$ topology, which is appropriate when clustering of large values do not occur.
When stochastic processes exhibit rapid successions of
jumps within temporal clusters of large values, collapsing in the limit to a single
jump, the $J_{1}$ topology become inappropriate since the $J_{1}$ convergence fails to hold. The next option is to use a weaker topology
in which the functional convergence may still hold, for example the Skorohod $M_{1}$ topology.
For definitions and discussion about the $J_{1}$ and $M_{1}$ topologies and the corresponding metrics $d_{J_{1}}$ and $d_{M_{1}}$ we refer to Resnick~\cite{Re07}, section 3.3.4 and Whitt~\cite{Whitt02}, sections 3.3 and 12.3--12.5.

Let $(X_{n})_{n \in \mathbb{Z}}$ be a strictly stationary sequence of nonnegative random variables and assume it is jointly regularly varying with index $\alpha >0$.
A standard procedure in obtaining functional limit theorems for maxima processes consists first in obtaining limit results for the corresponding point processes of jumps and then transferring this convergence to maxima processes. In order to establish this point process convergence, Basrak et al.~\cite{BKS} introduced the following time-space processes
\begin{equation*}\label{E:ppspacetime}
  N_{n} = \sum_{i=1}^{n} \delta_{(i / n,\,X_{i} / a_{n})} \qquad \textrm{for all} \ n\in \mathbb{N},
\end{equation*}
where $(a_{n})$ is a sequence of positive real numbers such that (\ref{eq:an}) holds.
They obtained weak convergence of $N_{n}$ in the state space $[0,1] \times \mathbb{E}_{u}$ for every $u>0$ under weak dependence Conditions~\ref{c:mixcond1} and~\ref{c:mixcond2} given below. When $N_{n}$ converges to a Poisson process, Tyran-Kami\'{n}ska~\cite{TK10} in Theorem 4.2 obtained some necessary conditions for this point process convergence in terms of partial maxima processes.

\begin{cond}\label{c:mixcond1}
There exists a sequence of positive integers $(r_{n})$ such that $r_{n} \to \infty $ and $r_{n} / n \to 0$ as $n \to \infty$ and such that for every $f \in C_{K}^{+}([0,1] \times \mathbb{E})$, denoting $k_{n} = \lfloor n / r_{n} \rfloor$, as $n \to \infty$,
\begin{equation}\label{e:mixcon}
 \E \biggl[ \exp \biggl\{ - \sum_{i=1}^{n} f \biggl(\frac{i}{n}, \frac{X_{i}}{a_{n}}
 \biggr) \biggr\} \biggr]
 - \prod_{k=1}^{k_{n}} \E \biggl[ \exp \biggl\{ - \sum_{i=1}^{r_{n}} f \biggl(\frac{kr_{n}}{n}, \frac{X_{i}}{a_{n}} \biggr) \biggr\} \biggr] \to 0.
\end{equation}
\end{cond}


\begin{cond}\label{c:mixcond2}
There exists a sequence of positive integers $(r_{n})$ such that $r_{n} \to \infty $ and $r_{n} / n \to 0$ as $n \to \infty$ and such that for every $u > 0$,
\begin{equation}
\label{e:anticluster}
  \lim_{m \to \infty} \limsup_{n \to \infty}
  \PP \biggl( \max_{m \leq |i| \leq r_{n}} X_{i} > ua_{n}\,\bigg|\,X_{0}>ua_{n} \biggr) = 0.
\end{equation}
\end{cond}
Condition~\ref{c:mixcond1} is implied by strong mixing, which we show in the proposition below. This is the so-called Condition $\mathcal{A}(a_{n})$ of Davis and Hsing~\cite{DaHs95}, and it seems that it has not been proved formally before. Recall that a sequence of random variables $(\xi_{n})$ is strongly mixing if $\alpha_{n} \to 0$ as $n \to \infty$, where
$$\alpha_{n} = \sup \{|\Pr (A \cap B) - \Pr(A) \Pr(B)| : A \in \mathcal{F}_{-\infty}^{j}, B \in \mathcal{F}_{j+n}^{\infty}, j=1,2, \ldots \}$$
and $\mathcal{F}_{k}^{l} = \sigma( \{ \xi_{i} : k \leq i \leq l \} )$ for $-\infty \leq k \leq l \leq \infty$.

\begin{proposition}\label{p:mixcond1}
Suppose $(X_{n})$ is a strictly stationary sequence of nonnegative regularly varying random variables with index $\alpha>0$. If $(X_{n})$ is strongly mixing then Condition~\ref{c:mixcond1} holds.
\end{proposition}
\begin{proof}
Let $(l_{n})$ be an arbitrary sequence of positive
integers such that $l_{n} \to \infty$ as $n \to \infty$ and
 $l_{n} = o(n^{1/8})$, where $b_{n}=o(c_{n})$ means $b_{n}/c_{n} \to 0$ as $n \to \infty$.
  Define, for any $n \in \mathbb{N}$,
 \begin{equation}\label{e:rn1}
 r_{n} = \lfloor \max \{ n \sqrt{ \alpha_{l_{n}+1}},\,n^{2/3} \} \rfloor
    +1.
 \end{equation}
 Then $r_{n} \to \infty$ as $n \to \infty$.
 Since the sequence $(X_{n})$ is strongly mixing,
  $\alpha_{l_{n}+1} \to 0$ as $n \to \infty$, and therefore
   $r_{n}/n \to 0$ as $n \to \infty$.
   Hence it follows that  $k_{n}
  \to \infty$ and
\begin{equation}\label{e:strmix1}
    k_{n} \alpha_{l_{n}+1} \to 0 \qquad \textrm{and} \qquad
    \frac{k_{n}l_{n}}{n} \to 0.
\end{equation}

Fix $f \in C_{K}^{+}([0,1] \times \mathbb{E})$. We have to show that
$I(n) \to 0$ as $n \to \infty$, where
$$ I(n) =  \bigg| \mathrm{E}  \bigg[ \exp \bigg\{ - \sum_{i=1}^{n} f \bigg(\frac{i}{n}, \frac{X_{i}}{a_{n}}
 \bigg) \bigg\} \bigg]
 - \prod_{k=1}^{k_{n}} \mathrm{E} \exp \bigg\{ - \sum_{i=1}^{r_{n}} f \bigg(\frac{kr_{n}}{n}, \frac{X_{i}}{a_{n}} \bigg)
 \bigg\} \bigg|.$$
 We have
\begin{eqnarray}\label{e:I}
  \nonumber I(n) & \leq & \bigg| \mathrm{E} \bigg[ \exp \bigg\{ - \sum_{i=1}^{n} f \bigg(\frac{i}{n}, \frac{X_{i}}{a_{n}}
 \bigg) \bigg\} \bigg]
 - \mathrm{E} \bigg[ \exp \bigg\{ - \sum_{i=1}^{k_{n}r_{n}} f \bigg(\frac{i}{n}, \frac{X_{i}}{a_{n}} \bigg)
 \bigg\} \bigg] \bigg| \\[0.5em]
   \nonumber & & \hspace*{-2.9em} + \ \bigg| \mathrm{E} \bigg[  \exp \bigg\{ - \sum_{i=1}^{k_{n}r_{n}} f \bigg(\frac{i}{n}, \frac{X_{i}}{a_{n}}
 \bigg) \bigg\} \bigg]
 - \mathrm{E} \bigg[ \exp \bigg\{ - \sum_{k=1}^{k_{n}} \sum_{i=(k-1)r_{n}+1}^{kr_{n}-l_{n}}
   f \bigg(\frac{i}{n}, \frac{X_{i}}{a_{n}} \bigg)
 \bigg\} \bigg] \bigg| \\[0.5em]
 \nonumber & &  \hspace*{-2.9em} + \ \bigg| \mathrm{E} \bigg[ \exp \bigg\{ - \sum_{k=1}^{k_{n}} \sum_{i=(k-1)r_{n}+1}^{kr_{n}-l_{n}}
    f \bigg(\frac{i}{n}, \frac{X_{i}}{a_{n}} \bigg) \bigg\} \bigg]
 - \prod_{k=1}^{k_{n}} \mathrm{E} \bigg[ \exp \bigg\{ - \sum_{i=1}^{r_{n}-l_{n}} f \bigg(\frac{kr_{n}}{n}, \frac{X_{i}}{a_{n}} \bigg)
 \bigg\} \bigg] \bigg| \\[0.5em]
   \nonumber & & \hspace*{-2.9em} + \ \bigg| \prod_{k=1}^{k_{n}} \mathrm{E} \bigg[ \exp \bigg\{ - \sum_{i=1}^{r_{n}-l_{n}} f \bigg(\frac{kr_{n}}{n}, \frac{X_{i}}{a_{n}}
 \bigg) \bigg\} \bigg]
 - \prod_{k=1}^{k_{n}} \mathrm{E} \bigg[ \exp \bigg\{ - \sum_{i=1}^{r_{n}} f \bigg(\frac{kr_{n}}{n}, \frac{X_{i}}{a_{n}} \bigg)
 \bigg\} \bigg] \bigg| \\[1em]
   & =: & I_{1}(n) + I_{2}(n) + I_{3}(n) + I_{4}(n)
\end{eqnarray}

 The function $f$ is nonnegative, bounded (by $M>0$ let us suppose)
 and its support is bounded away from origin, which implies that $f(s,x)=0$ for all $s \in
 [0,1]$ and $ x \in (0, \delta]$ for some $\delta >0$. Denote by
 $j_{n}=n-k_{n}r_{n}$. Then using stationarity and the inequality
 $1-e^{-x} \leq x$ for any $x \geq 0$, we obtain
 \begin{eqnarray}\label{e:I1}
   \nonumber I_{1}(n) & \leq & \mathrm{E} \bigg[ \exp \bigg\{ - \sum_{i=1}^{k_{n}r_{n}}
                                    f \bigg(\frac{i}{n}, \frac{X_{i}}{a_{n}} \bigg) \bigg\} \cdot \bigg|
                                    1 - \exp \bigg\{ - \sum_{i=k_{n}r_{n}+1}^{n}
                                    f \bigg(\frac{i}{n}, \frac{X_{i}}{a_{n}} \bigg) \bigg\} \bigg| \bigg]
                                    \\[0.5em]
   \nonumber & \leq & \mathrm{E} \bigg[ \sum_{i=k_{n}r_{n}+1}^{n} f \bigg(\frac{i}{n}, \frac{X_{i}}{a_{n}}
                   \bigg) \bigg] =   \sum_{i=k_{n}r_{n}+1}^{n} \mathrm{E} \bigg[ f \bigg(\frac{i}{n}, \frac{X_{1}}{a_{n}} \bigg)
                   1_{\big\{ \frac{|X_{1}|}{a_{n}} > \delta \big\} }
                   \bigg] \\[0.6em]
   & \leq & Mj_{n} \mathrm{P} (X_{1}>\delta a_{n}).
 \end{eqnarray}
In a similar manner we obtain
\begin{equation}\label{e:I2}
    I_{2}(n) \leq M k_{n}l_{n} \mathrm{P} (X_{1}>\delta a_{n}).
\end{equation}
Further we have
$$ I_{3}(n) \leq I_{5}(n) + I_{6}(n) + I_{7}(n),$$
where
\begin{eqnarray*}
I_{5}(n) & = & \bigg| \mathrm{E} \bigg[ \exp \bigg\{ - \sum_{k=1}^{k_{n}} \sum_{i=(k-1)r_{n}+1}^{kr_{n}-l_{n}}
      f \bigg( \frac{i}{n}, \frac{X_{i}}{a_{n}} \bigg) \bigg\} \bigg] \\[0.3em]
   & &  \hspace*{-2.2em} - \ \mathrm{E} \bigg[ \exp \bigg\{ - \sum_{i=1}^{r_{n}-l_{n}} f \bigg(\frac{i}{n}, \frac{X_{i}}{a_{n}} \bigg) \bigg\} \bigg]
     \mathrm{E} \bigg[ \exp \bigg\{ - \sum_{k=2}^{k_{n}} \sum_{i=(k-1)r_{n}+1}^{kr_{n}-l_{n}}
    f \bigg(\frac{i}{n}, \frac{X_{i}}{a_{n}} \bigg) \bigg\} \bigg] \bigg|,
\end{eqnarray*}
\begin{eqnarray*}
I_{6}(n) & = & \bigg| \mathrm{E} \bigg[ \exp \bigg\{ - \sum_{i=1}^{r_{n}-l_{n}} f \bigg(\frac{i}{n}, \frac{X_{i}}{a_{n}} \bigg) \bigg\} \bigg]
     \mathrm{E} \bigg[ \exp \bigg\{ - \sum_{k=2}^{k_{n}} \sum_{i=(k-1)r_{n}+1}^{kr_{n}-l_{n}}
     f \bigg(\frac{i}{n}, \frac{X_{i}}{a_{n}} \bigg) \bigg\} \bigg]\\[0.3em]
   & & \hspace*{-2.2em} - \ \mathrm{E} \bigg[ \exp \bigg\{ - \sum_{i=1}^{r_{n}-l_{n}} f \bigg(\frac{1 \cdot r_{n}}{n}, \frac{X_{i}}{a_{n}} \bigg) \bigg\} \bigg]
     \mathrm{E} \bigg[ \exp \bigg\{ - \sum_{k=2}^{k_{n}} \sum_{i=(k-1)r_{n}+1}^{kr_{n}-l_{n}}
     f \bigg(\frac{i}{n}, \frac{X_{i}}{a_{n}} \bigg) \bigg\} \bigg]\bigg|,
\end{eqnarray*}
and
\begin{eqnarray*}
I_{7}(n) & = & \bigg| \mathrm{E} \bigg[ \exp \bigg\{ - \sum_{i=1}^{r_{n}-l_{n}} f \bigg(\frac{1 \cdot r_{n}}{n}, \frac{X_{i}}{a_{n}} \bigg) \bigg\} 
     \mathrm{E} \bigg[ \exp \bigg\{ - \sum_{k=2}^{k_{n}} \sum_{i=(k-1)r_{n}+1}^{kr_{n}-l_{n}}
     f \bigg(\frac{i}{n}, \frac{X_{i}}{a_{n}} \bigg) \bigg\} \bigg]\\[0.3em]
  & & \hspace*{-2.2em} - \ \prod_{k=1}^{k_{n}} \mathrm{E} \bigg[ \exp \bigg\{ - \sum_{i=1}^{r_{n}-l_{n}}
      f \bigg(\frac{kr_{n}}{n}, \frac{X_{i}}{a_{n}} \bigg) \bigg\} \bigg]\bigg|.
\end{eqnarray*}
The inequality
$ | \mathrm{E} (gh) - \mathrm{E} g\,\mathrm{E} h| \leq 4C_{1}C_{2} \alpha_{m}$,
for a $\mathcal{F}_{-\infty}^{j}$ measurable function $g$ and a
$\mathcal{F}_{j+m}^{\infty}$ measurable function $h$ such that $|g|
\leq C_{1}$ and $|h| \leq C_{2}$ (see Lemma
1.2.1 in Lin and Lu~\cite{LiLu97}), gives
\begin{equation}\label{e:I5}
    I_{5}(n) \leq 4 \alpha_{l_{n}+1}.
\end{equation}
For any $t>0$ there exists a constant $C(t)>0$ such that the
following inequality holds:
$$ |1-e^{-x}| \leq C(t)|x| \qquad \textrm{for all} \ |x| \leq t.$$
Further, for arbitrary real numbers $z_{1}, \ldots, z_{n}$ and $w_{1}, \ldots, w_{n}$ it holds that
\begin{equation}\label{eq:durrett}
\bigg| \prod_{k=1}^{n}z_{k} - \prod_{k=1}^{n}w_{k} \bigg| \leq A^{n-1} \sum_{k=1}^{n}|z_{k}-w_{k}|
\end{equation}
where $A= \max \{|z_{1}|,\ldots, |z_{n}|, |w_{1}|,\ldots, |w_{n}| \}$.
These last two inequalities imply
\begin{eqnarray*}
  I_{6}(n) & \leq & \mathrm{E} \bigg| \exp \bigg\{ - \sum_{i=1}^{r_{n}-l_{n}}
       f \bigg(\frac{i}{n}, \frac{X_{i}}{a_{n}} \bigg) \bigg\} - \exp \bigg\{ - \sum_{i=1}^{r_{n}-l_{n}}
       f \bigg(\frac{r_{n}}{n}, \frac{X_{i}}{a_{n}} \bigg) \bigg\} \bigg|
       \\[0.3em]
   & \leq &  \sum_{i=1}^{r_{n}-l_{n}} \mathrm{E} \bigg| \exp \bigg\{ -f \bigg(\frac{i}{n}, \frac{X_{i}}{a_{n}} \bigg) \bigg\}
       - \exp \bigg\{ -f \bigg(\frac{r_{n}}{n}, \frac{X_{i}}{a_{n}} \bigg) \bigg\}
       \bigg|\\[0.3em]
   & \leq & \sum_{i=1}^{r_{n}-l_{n}} \mathrm{E} \bigg| 1
       - \exp \bigg\{ f \bigg(\frac{i}{n}, \frac{X_{i}}{a_{n}} \bigg) -
       f \bigg(\frac{r_{n}}{n}, \frac{X_{i}}{a_{n}} \bigg) \bigg\}
       \bigg|\\[0.3em]
   & \leq & C(2M) \sum_{i=1}^{r_{n}-l_{n}} \mathrm{E} \bigg| f \bigg(\frac{i}{n}, \frac{X_{i}}{a_{n}} \bigg) -
       f \bigg(\frac{r_{n}}{n}, \frac{X_{i}}{a_{n}} \bigg)
       \bigg|.
\end{eqnarray*}
Therefore
\begin{eqnarray*}
  I_{6}(n) & \leq & C(2M) \sum_{i=1}^{r_{n}-l_{n}} \mathrm{E} \bigg[ \bigg| f \bigg(\frac{i}{n}, \frac{X_{i}}{a_{n}} \bigg) -
       f \bigg(\frac{r_{n}}{n}, \frac{X_{i}}{a_{n}} \bigg)
       \bigg| 1_{ \big\{ \frac{X_{i}}{a_{n}} > \delta \big\} } \bigg].
\end{eqnarray*}
Since a continuous function on a compact set is uniformly
continuous, it follows that for any $\epsilon >0$ there exists
$\gamma >0$ such that for $(s,x), (s',x') \in [0,1] \times \{y \in
\mathbb{E} : y > \delta \}$, if $d_{[0,1] \times
\mathbb{E}}((s,x),(s',x')) < \gamma$ then $|f(s,x) - f(s',x')|<
\epsilon$, where by $d_{[0,1] \times \mathbb{E}}$ we denoted the
metric on the direct product of metric spaces $[0,1]$ and
$\mathbb{E}$, i.e.
$ d_{[0,1] \times \mathbb{E}}((s,x),(s',x')) = \max \{ |s-s'|,
\rho(x,x') \}.$ Since $r_{n}/n \to 0$ as $n \to \infty$,
for large $n$ we have
$$ d_{[0,1] \times \mathbb{E}} \bigg( \bigg(\frac{i}{n}, \frac{X_{i}}{a_{n}}\bigg),
  \bigg(\frac{r_{n}}{n}, \frac{X_{i}}{a_{n}}\bigg) \bigg) = \frac{|i
   -r_{n}|}{n} \leq \frac{r_{n}}{n} < \gamma,$$
for any $i=1, \ldots , r_{n}-l_{n}$. Therefore, for large $n$,
$$ \bigg| f \bigg(\frac{i}{n}, \frac{X_{i}}{a_{n}} \bigg) -
       f \bigg(\frac{r_{n}}{n}, \frac{X_{i}}{a_{n}} \bigg)
       \bigg| < \epsilon,$$ and this implies
\begin{equation}\label{e:I6}
  I_{6}(n) \leq  \epsilon\,C(2M) (r_{n}-l_{n}) \mathrm{P} ( X_{1} > \delta
  a_{n}) \qquad \textrm{for large} \ n.
\end{equation}
Taking into account relations (\ref{e:I5}) and (\ref{e:I6}), it
follows that, for large $n$,
$$ I_{3}(n) \leq 4 \alpha_{l_{n}+1} + \epsilon\,C(2M) r_{n}
\mathrm{P}(X_{1} > \delta a_{n}) + I_{7}(n),$$
 and since it is easy to obtain
\begin{eqnarray*}
 I_{7}(n) & \leq & \bigg|
     \mathrm{E} \bigg[ \exp \bigg\{ - \sum_{k=2}^{k_{n}} \sum_{i=(k-1)r_{n}+1}^{kr_{n}-l_{n}}
     f \bigg(\frac{i}{n}, \frac{X_{i}}{a_{n}} \bigg) \bigg\} \bigg]\\[0.5em]
      & &
-\,\prod_{k=2}^{k_{n}} \mathrm{E}
      \bigg[ \exp \bigg\{ - \sum_{i=1}^{r_{n}-l_{n}}
      f \bigg(\frac{kr_{n}}{n}, \frac{X_{i}}{a_{n}} \bigg) \bigg\} \bigg]
      \bigg|,
\end{eqnarray*}
we recursively obtain (we repeat the same procedure for $I_{7}(n)$
as we did for $I_{3}(n)$ and so on)
\begin{equation}\label{e:I3}
    I_{3}(n) \leq 4 k_{n} \alpha_{l_{n}+1} + \epsilon\,C(2M)
    k_{n}r_{n} \mathrm{P} (X_{1}>\delta a_{n}).
\end{equation}
Stationarity and (\ref{eq:durrett}) imply
\begin{equation}\label{e:I4}
    I_{4}(n) \leq M k_{n}l_{n} \mathrm{P} (X_{1} >\delta a_{n}).
\end{equation}
Thus from (\ref{e:I}), (\ref{e:I1}), (\ref{e:I2}),
(\ref{e:I3}) and (\ref{e:I4}) it follows that for large $n$,
\begin{eqnarray*}
  I(n) & \leq & \bigg( M \frac{j_{n}}{n} + 2M \frac{k_{n}l_{n}}{n} +
         \epsilon\,C(2M) \frac{k_{n}r_{n}}{n} \bigg) \cdot n \mathrm{P}(X_{1}>
         a_{n}) \cdot \frac{\mathrm{P}(X_{1}> \delta a_{n})}{\mathrm{P}(X_{1}> a_{n})}
         \\[0.5em]
   & & + \ 4 k_{n} \alpha_{l_{n}+1}.
\end{eqnarray*}
 Since $X_{1}$ is regularly varying, it holds
 that
 $ \mathrm{P}(X_{1}> \delta a_{n})\,/\,\mathrm{P}(X_{1}> a_{n})
 \rightarrow \delta^{-\alpha},$
 as $n \rightarrow \infty$.
 This together with relation
  (\ref{e:strmix1}), and the fact that $j_{n}/n \rightarrow 0$, $k_{n}r_{n}/n \rightarrow 1$ and
  $n \mathrm{P}(X_{1}>a_{n}) \to 1$ as
 $n \rightarrow \infty$, imply
$$ \limsup_{n \rightarrow \infty} I(n) \leq \epsilon\,C(2M)
    \delta^{-\alpha}.$$
But since this holds for all $\epsilon >0$, we get $ \lim_{n
\rightarrow \infty} I(n) = 0$, and thus Condition~\ref{c:mixcond1}
 holds.\qed
\end{proof}

Under the finite-cluster Condition~\ref{c:mixcond2} the following value
\begin{equation}\label{e:theta}
   \theta  = \lim_{r \to \infty} \lim_{x \to \infty} \PP \biggl(\max_{1 \leq i \leq r}X_{i} \leq x \, \bigg| \, X_{0}>x \biggr)
\end{equation}
is strictly positive, and it is equal to the extremal index of the sequence $(X_{n})$ (see Basrak and Segers~\cite{BaSe}). For a definition and some discussion about the extremal index we refer to Leadbetter and Rootz\'{e}n~\cite{LeRo88}, page 439.

\begin{proposition}\label{p:mixcond2jrv}
 Let $(X_{n})$ be a strictly stationary sequence of nonnegative regularly varying random variables with index $\alpha>0$. If $(X_{n})$ is strongly mixing and has extremal index $\theta=1$, then:
 \begin{itemize}
   \item[(i)] Condition~\ref{c:mixcond2} holds.
   \item[(ii)] The sequence $(X_{n})$ is jointly regularly varying with index $\alpha$. Further, for the tail process $(Y_{n})$ of $(X_{n})$ it holds that $Y_{k}=0$ for all $k \neq 0$.
   \item[(iii)] The following point process convergence holds
    \begin{equation}\label{e:Nnconv}
N_{n} \bigg|_{[0, 1] \times \EE_u} \dto N^{(u)}
    = \sum_i \delta_{(T^{(u)}_i, u Y_{0})} \bigg|_{[0, 1] \times \EE_u}\,
\end{equation}
in $[0, 1] \times \EE_u$ for every $u \in (0,\infty)$, where $\sum_i \delta_{T^{(u)}_i}$ is a
homogeneous Poisson process on $[0, 1]$ with intensity $u^{-\alpha}$.
 \end{itemize}
\end{proposition}
\begin{proof}
(i) Let $(q_{n})$ be any sequence of positive integers such that $q_{n} \to \infty$ and $q_{n} = o(n)$.
Fix an arbitrary $u>0$ and put
\begin{equation}\label{e:rn2}
p_{n}= \max \{ \floor{n \sqrt{\alpha_{q_{n}}}}, \floor{ \sqrt{n q_{n}}} +1 \},
\end{equation}
 where $(\alpha_{n})$ is the sequence of $\alpha$--mixing coefficients of $(X_{n})$. From Theorem 2.1 and Proposition 5.1 in O'Brien~\cite{O'B87} we derive that, as $n \to \infty$,
 \begin{equation}\label{e:O'Brien1}
 \PP ( M_{n} \leq ua_{n} ) - [\PP (X_{0} \leq ua_{n})]^{t_{n}} \to 0,
 \end{equation}
where $t_{n} = n \PP ( M_{p_{n}} \leq ua_{n}\,|\,X_{0}> ua_{n} )$.

Let $(\widehat{X}_{n})$ be the associated independent sequence of $(X_{n})$, i.e. $(\widehat{X}_{n})$ is an i.i.d. sequence with $\widehat{X}_{1} \eqd X_{1}$, and let $\widehat{M}_{n} = \max \{\widehat{X}_{i} : i=1,\ldots, n \}$. Then by Theorem 2.2.1 in Leadbetter and Rootz\'{e}n~\cite{LeRo88}
\begin{equation}\label{e:LRpower}
 \PP ( M_{n} \leq u a_{n} ) \to G^{\theta}(u) \qquad \textrm{as} \ n \to \infty,
\end{equation}
where
$$ G(u) = \lim_{n \to \infty} \PP ( \widehat{M}_{n} \leq u a_{n} )  = \lim_{n \to \infty} [ \PP (\widehat{X}_{0} \leq ua_{n}) ]^{n} = e^{-u^{-\alpha}}.$$
Since $\theta=1$, from (\ref{e:LRpower}) we obtain
\begin{equation}\label{e:LRpower2}
\PP ( M_{n} \leq u a_{n}) \to e^{-u^{-\alpha}} \qquad \textrm{as} \ n \to \infty.
\end{equation}
Therefore, from (\ref{e:O'Brien1}) and (\ref{e:LRpower2}) we obtain, as $n \to \infty$,
$$ \PP ( M_{p_{n}} \leq ua_{n}\,|\,X_{0}> ua_{n} ) \cdot \ln [ \PP(X_{0} \leq u a_{n} ) ]^{n} \to - u^{-\alpha},$$
and since $\lim_{n \to \infty} [ \PP (X_{0} \leq ua_{n})]^{n} = e^{-u^{-\alpha}}$, it follows that
\begin{equation}\label{e:forjrv}
 \PP ( M_{p_{n}} > ua_{n}\,|\,X_{0} > ua_{n} ) \to 0 \qquad \textrm{as} \ n \to \infty.
\end{equation}
From this, putting $r_{n} :=p_{n}$, we deduce that
\begin{equation*}
\label{e:anticlusterpos}
  \lim_{m \to \infty} \limsup_{n \to \infty}
  \PP \biggl( \max_{m \leq i \leq r_{n}} X_{i} > ua_{n}\,\bigg|\,X_{0}>ua_{n} \biggr) = 0.
\end{equation*}
One similarly deals with negative indices, and hence we finally conclude
\begin{equation*}
  \lim_{m \to \infty} \limsup_{n \to \infty}
  \PP \biggl( \max_{m \leq |i| \leq r_{n}} X_{i} > ua_{n}\,\bigg|\,X_{0}>ua_{n} \biggr) = 0.
\end{equation*}

(ii) From relation (\ref{e:forjrv}) one straightforward obtains that for all $k \neq 0$ and $r \in (0,1)$
$$ \lim_{n \to \infty} \PP(X_{k} > r a_{n}\,|\,X_{0} > a_{n}) = 0,$$
which implies $P(Y_{k} > r) = 0$, i.e. $Y_{k}=0$. As for $Y_{0}$, from (\ref{e:tailprocess}) and the regular variation property of $X_{0}$ we immediately obtain
$\PP(Y_{0} > y)=y^{-\alpha}$ for $y \geq 1$. These suffices to conclude that $(X_{n})$ is jointly regularly varying.

(iii) Since $(X_{n})$ is jointly regularly varying and Conditions \ref{c:mixcond1} and~\ref{c:mixcond2} hold (note that Condition \ref{c:mixcond1} holds by Proposition \ref{p:mixcond1}), by Theorem 2.3 in Basrak and Segers~\cite{BKS}, for every $u \in (0,\infty)$ and as $n \to \infty$,
\begin{equation*}
N_{n} \bigg|_{[0, 1] \times \EE_u} \dto N^{(u)}
    = \sum_i \sum_j \delta_{(T^{(u)}_i, u Z_{ij})} \bigg|_{[0, 1] \times \EE_u}\,
\end{equation*}
in $[0, 1] \times \EE_u$, where $(\sum_j \delta_{Z_{ij}})_i$ is an i.i.d.\
sequence of point processes in $\EE$, independent of $\sum_i
\delta_{T^{(u)}_i}$, and with common distribution equal to the
distribution of
$$\biggl( \sum_{n \in \mathbb{Z}} \delta_{Y_n} \, \bigg| \, \sup_{i \leq -1} Y_i \leq 1 \biggr).$$
From this, since $Y_{k}=0$ for $k \neq 0$, we immediately obtain (\ref{e:Nnconv}).
\qed
\end{proof}

\section{Functional $J_{1}$ convergence of partial maxima processes}\label{s:two}

Let $(X_n)$ be a strongly mixing and strictly stationary sequence of nonnegative regularly varying random variables with index $\alpha >0$.
In this section we show the convergence of the partial maxima
processes $M_n(\,\cdot\,)$ to an extremal process in the space $D[0, 1]$
equipped with the Skorohod $J_1$ topology when there is no clustering of large values in the sequence $(X_{n})$. Similar to the case of
partial sum processes in Basrak and Segers~\cite{BKS} we first represent
the partial maxima process $M_n(\,\cdot\,)$ as the image of the
time-space point process $N_{n}\,|\,_{[0,1] \times \EE_{u}}$ under a certain maximum
functional. Then, using certain continuity properties of this
functional, the continuous mapping theorem and the standard
"finite dimensional convergence plus tightness" procedure we
transfer the weak convergence of $N_{n}\,|\,_{[0,1] \times \EE_{u}}$ in to
weak convergence of $M_n(\,\cdot\,)$.

Extremal processes can be defined by Poisson processes in the following way. Let $\xi = \sum_{k}\delta_{(t_{k}, j_{k})}$ be a Poisson process on $(0,\infty) \times \mathbb{E}$ with mean measure $\lambda \times \nu$, where $\lambda$ is the Lebesgue measure. The extremal process $\widetilde{M}(\,\cdot\,)$ generated by $\xi$ is defined by
$ \widetilde{M}(t) = \sup \{ j_{k} : t_{k} \leq t\},\,t>0.$
The distribution function of $\widetilde{M}(t)$ is of the form
$$ \PP ( \widetilde{M}(t) \leq x) = e^{-t \nu(x,\infty)}$$
for $t>0$ (cf. Resnick~\cite{Re86}). The measure $\nu$ is called the exponent measure.

Fix $0 < v < u < \infty$. Define the maximum functional
$
  \phi^{(u)} \colon \mathbf{M}_{p}([0,1] \times \EE_{v}) \to
  D[0,1]
$
 by
 $$ \phi^{(u)} \Big( \sum_{i}\delta_{(t_{i},\,x_{i})} \Big) (t)
  =  \bigvee_{t_{i} \leq t} x_{i} \,1_{\{u < x_i < \infty\}}, \qquad t \in [0,
  1],$$
  where the supremum of an empty set may be taken, for
  convenience, to be $0$.
The space $\mathbf{M}_p([0,1] \times \EE_{v})$ of Radon point
measures on $[0,1] \times \EE_{v}$ is equipped with the vague
topology and $D[0, 1]$ is equipped with the $J_1$ topology. Let
$\Lambda = \Lambda_{1} \cap \Lambda_{2}$ where
\begin{eqnarray*}
  \Lambda_{1} &=& \{ \eta \in \mathbf{M}_{p}([0,1] \times \EE_{v}) :
    \eta ( \{0,1 \} \times \EE_{u}) = \eta ([0,1] \times \{u, \infty \})=0 \}, \\[0.3em]
  \Lambda_{2} &=& \{ \eta \in \mathbf{M}_{p}([0,1] \times \EE_{v}) :
    \eta ( \{ t \} \times \EE_{v}) \leq 1 \
   \text{for all $t \in [0,1]$} \}.
\end{eqnarray*}
 Then the point process $N^{(v)}$ defined in
(\ref{e:Nnconv}) almost surely belongs to the set $\Lambda$, see
Lemma 3.1 in Basrak et al.~\cite{BKS}. With similar arguments as in Resnick~\cite{Re07}, pages 224--226, one obtains the following lemma.

\begin{lemma}\label{l:contfunct}
The maximum functional $\phi^{(u)} \colon \mathbf{M}_{p}([0,1]
\times \EE_{v}) \to D[0,1]$ is continuous on the set $\Lambda$,
when $D[0,1]$ is endowed with the Skorohod $J_{1}$ topology.
\end{lemma}

Now we are ready to prove the functional $J_{1}$ convergence of partial maxima processes.

\begin{theorem}\label{t:functconvergence}
Let $(X_{n})$ be a strictly stationary sequence of nonnegative regularly varying
random variables with index $\alpha >0$. Suppose the sequence $(X_{n})$ is strongly mixing and has extremal index $\theta=1$. Then the partial maxima stochastic process
$$  M_{n}(t) = \bigvee_{i=1}^{\lfloor nt
   \rfloor}\frac{X_{i}}{a_{n}}, \qquad t \in [0,1],$$
satisfies
$ M_{n}(\,\cdot\,) \dto \widetilde{M}(\,\cdot\,)$ as $n \to \infty,$
in $D[0,1]$ endowed with the $J_{1}$ topology, where
$\widetilde{M}(\,\cdot\,)$ is an extremal process with exponent
measure $\nu(x,\infty) = x^{-\alpha}$, $x>0$.
\end{theorem}

\begin{remark}
The statement of Theorem~\ref{t:functconvergence} is very similar to the statement of Theorem 4.1 in Krizmani\'{c}~\cite{Kr14}, with the difference that in the case treated in Krizmani\'{c}~\cite{Kr14} there is no restriction on the extremal index and the convergence takes place in $D[0,1]$ with the $M_{1}$ topology. The restriction on the extremal index (i.e. $\theta=1$) in Theorem~\ref{t:functconvergence} allows us to obtain the convergence of the partial maxima process in the stronger $J_{1}$ topology. Since the proof of Theorem~\ref{t:functconvergence} follows closely the lines of the proof of Theorem 4.1 in Krizmani\'{c}~\cite{Kr14} we will omit those parts that are identical. The only differences that occur are those arguments that use the notion of the $J_{1}$ instead of the $M_{1}$ topology, and we will describe them in the following proof.
\end{remark}

\begin{remark}
In the proof below we will use the convergence result (\ref{e:Nnconv}) for point processes $N_{n}$. For this Conditions~\ref{c:mixcond1} and~\ref{c:mixcond2} must hold, but with the same sequence $(r_{n})$. If Condition~\ref{c:mixcond1} holds with the sequence $(r_{n}^{(1)})$ as given in (\ref{e:rn1}) and Condition~\ref{c:mixcond2} holds with the sequence $(r_{n}^{(2)})$ as given in (\ref{e:rn2}), by letting $r_{n} = r_{n}^{(1)} \vee r_{n}^{(2)}$, we can repeat all the arguments from the proofs of Propositions~\ref{p:mixcond1} and~\ref{p:mixcond2jrv}, and hence obtain that Conditions~\ref{c:mixcond1} and~\ref{c:mixcond2} both hold with the sequence $(r_{n})$.
\end{remark}

\begin{proof} (Theorem~\ref{t:functconvergence})
 Consider $0<u<v$ and
$$  \phi^{(u)} (N_{n}\,|\,_{[0,1] \times \EE_{u}}) (\,\cdot\,)
  = \phi^{(u)} (N_{n}\,|\,_{[0,1] \times \EE_{v}}) (\,\cdot\,)
  = \bigvee_{i/n \leq \, \cdot} \frac{X_{i}}{a_{n}} 1_{ \big\{ \frac{X_{i}}{a_{n}} > u
    \big\} },$$
which by (\ref{e:Nnconv}), Lemma~\ref{l:contfunct} and the continuous mapping theorem converges in distribution under the $J_{1}$ metric to
$
\phi^{(u)}
(N^{(v)})(\,\cdot\,)
 =\phi^{(u)} (N^{(v)}\,|\,_{[0,1] \times \EE_{u}})(\,\cdot\,).
$
Using the arguments from the proof of Theorem 4.1 in Krizmani\'{c}~\cite{Kr14} this can be rewritten as
 \begin{equation}\label{eq:convaboveu}
  M_{n}^{(u)}(\,\cdot\,) := \bigvee_{i = 1}^{\lfloor n \, \cdot \, \rfloor} \frac{X_{i}}{a_{n}} 1_{ \big\{ \frac{X_{i}}{a_{n}} > u
    \big\} } \dto M^{(u)}(\,\cdot\,) := \bigvee_{T_{i} \leq \, \cdot} K_{i}^{(u)} \quad \text{as} \ n \to \infty,
 \end{equation}
 in $D[0,1]$ under the $J_{1}$ metric, where
 $ \widetilde{N}^{(u)} = \sum_{i} \delta_{(T_{i},\,K_{i}^{(u)})}
 $
 is a Poisson process with mean measure $\lambda \times \nu^{(u)}$ and
 $$  \nu^{(u)}(x, \infty) = u^{-\alpha} \PP( uY_0 > x) , \qquad x >0.$$ Note that
$ \nu^{(u)}(dx) = \alpha x^{-\alpha -1}1_{(u,\infty)}(x)\,dx.$

The limiting process $M^{(u)}(\,\cdot\,)$ is an extremal process with exponent measure $\nu^{(u)}$, since for $t \in [0,1]$ and $x >0$
$$ \PP (M^{(u)}(t) \leq x) = \PP(\widetilde{N}^{(u)}((0,t] \times (x,\infty))=0) = e^{-t\nu^{(u)}(x,\infty)}.$$

Now as in the proof of Theorem 4.1 in Krizmani\'{c}~\cite{Kr14} one shows that, as $u \to 0$, the finite dimensional distributions of $M^{(u)}(\,\cdot\,)$ converge to the finite dimensional distributions of an extremal process $\widetilde{M}(\,\cdot\,)$ generated by a Poisson process $\sum_{i}\delta_{(T_{i},K_{i})}$ with mean measure $\lambda \times \nu$, i.e. $\widetilde{M}(t) = \bigvee_{T_{i} \leq t} K_{i}$, $t \in [0,1]$.

Since we obtained convergence of finite dimensional distributions, in order to obtain $J_{1}$ convergence of $M^{(u)}(\,\cdot\,)$ to $\widetilde{M}(\,\cdot\,)$ as $u \to 0$, according to the well known result regarding tightness with respect to the $J_{1}$ topology (see Theorems 3.2.1 and 3.2.2 in Skorohod~\cite{Sk56}) we need only to show
$$\lim_{\delta \to 0} \limsup_{u \to 0} \PP ( \omega_{\delta}'(M^{(u)}(\,\cdot\,)) > \epsilon ) =0,$$
for every $\epsilon >0$, where $\omega_{\delta}'(x)$ is the $J_{1}$ oscillation of $x \in D[0,1]$, i.e.
$$ \omega_{\delta}'(x) = \sup_{{\footnotesize \begin{array}{c}
                                t_{1} \leq t \leq t_{2} \\
                                0 \leq t_{2}-t_{1} \leq \delta
                              \end{array}}
} \min\{| x(t)-x(t_{1})|, |x(t_{2})-x(t)| \},$$
for $\delta >0$.
Fix $\epsilon >0$ and take $u \in (0, \epsilon)$.
We can represent
\begin{equation}\label{e:repr}
 \widetilde{N}^{(u)}(([0,1] \times \EE_{\epsilon}) \cap\,\cdot\,) = \sum_{i=1}^{\xi}\delta_{(U_{i},V_{i}^{(u)})}(\,\cdot\,),
\end{equation}
where
$U_{1}, U_{2}, \ldots$ are i.i.d. uniformly distributed on $(0,1)$, $V_{1}^{(u)}, V_{2}^{(u)}, \ldots$ are i.i.d. with distribution $\nu^{(u)}(\EE_{\epsilon} \cap\,\cdot\,) / \nu^{(u)}(\EE_{\epsilon})$, and $\xi$ is a Poisson random variable with parameter $s:=(\lambda \times \nu^{(u)})([0,1] \times \EE_{\epsilon})= \nu^{(u)}((\epsilon, \infty))$ and independent of $(U_{i}, V_{i}^{(u)})_{i \geq 1}$ (cf. Resnick~\cite{Re07}, page 147).
Since $u < \epsilon$ we obtain $s=\epsilon^{-\alpha}$.

Note that $\omega_{\delta}'(M^{(u)}(\,\cdot\,)) > \epsilon$
 implies the existence of $t_{1} \leq t \leq t_{2}$ such that $0 \leq t_{2}-t_{1} \leq \delta$, $M^{(u)}(t) - M^{(u)}(t_{1}) > \epsilon$ and $M^{(u)}(t_{2}) - M^{(u)}(t)> \epsilon$, i.e.
$$\bigvee_{t_{1} \leq T_{i} \leq t} K_{i}^{(u)} > \epsilon \quad \textrm{and} \quad \bigvee_{t < T_{i} \leq t_{2}}K_{i}^{(u)} > \epsilon.$$
 Therefore there exist $T_{i} \in (t_{1}, t]$ and $T_{j} \in (t, t_{2}]$ such that
$K_{i}^{(u)}>\epsilon$ and $K_{j}^{(u)}>\epsilon$. This means that $M^{(u)}$ has (at least) two jumps on the set $(t_{1}, t_{2}]$ greater than $\epsilon$, i.e.
$\widetilde{N}^{(u)}((t_{1}, t_{2}] \times \EE_{\epsilon}) \geq 2$. Using the representation in (\ref{e:repr}) we get
$$ \sum_{i=1}^{\xi}\delta_{(U_{i},V_{i}^{(u)})}((t_{1}, t_{2}] \times \EE_{\epsilon}) \geq 2.$$
Therefore
\begin{eqnarray*}
   \PP ( \omega_{\delta}'(M^{(u)}(\,\cdot\,)) > \epsilon ) & \leq & \PP \bigg( \bigcup_{1 \leq i < j \leq \xi} \{|U_{i}-U_{j}| \leq \delta \} \bigg)\\[0.3em]
   &=&  \sum_{n=0}^{\infty} \PP \bigg(  \bigcup_{1 \leq i < j \leq n} \{|U_{i}-U_{j}| \leq \delta \} \bigg)\,\PP(\xi=n)\\[0.3em]
   & \leq & \sum_{n=0}^{\infty} {n \choose 2} \PP (|U_{1}-U_{2}| \leq \delta)\,e^{-s}\frac{s^{n}}{n!}.
\end{eqnarray*}
Since random variables $U_{i}$ are uniformly distributed on $(0,1)$ by standard calculations we get $\PP(|U_{1}-U_{2}| \leq \delta)=\delta (2-\delta)$ for $\delta <1$ (and obviously $\PP(|U_{1}-U_{2}| \leq \delta)= 1$ for $\delta  \geq 1$). Thus for $\delta < 1$ it holds that
\begin{eqnarray*}
   \PP ( \omega_{\delta}'(M^{(u)}(\,\cdot\,)) > \epsilon ) & \leq &  \delta (2-\delta)e^{-s}\frac{s^{2}}{2} \sum_{n=2}^{\infty} \frac{s^{n-2}}{(n-2)!}
    =  \delta (2-\delta) \frac{s^{2}}{2},
\end{eqnarray*}
 and this yields
$$\lim_{\delta \to 0} \limsup_{u \to 0} \PP ( \omega_{\delta}'(M^{(u)}(\,\cdot\,)) > \epsilon ) =0.$$
Therefore
\begin{equation}\label{eq:convtoMtilda}
M^{(u)}(\,\cdot\,) \dto \widetilde{M}(\,\cdot\,) \qquad \text{as} \ u \to 0,
\end{equation}
in $D[0,1]$ with the $J_{1}$ topology.

With the same arguments as in the proof of Theorem 4.1 in Krizmani\'{c}~\cite{Kr14} one shows that
$$ \lim_{u \to 0}\limsup_{n \to \infty} \PP(d_{J_{1}}(M_{n}(\,\cdot\,),M_{n}^{(u)}(\,\cdot\,)) > \epsilon)=0.$$
This with (\ref{eq:convaboveu}) and (\ref{eq:convtoMtilda}), according to a variant of Slutsky's theorem (see Theorem 3.5 in Resnick~\cite{Re07}), allows us to conclude that, as $n \to \infty$,
$ M_{n}(\,\cdot\,) \dto \widetilde{M}(\,\cdot\,)$,
in $D[0,1]$ with the $J_{1}$ topology.\qed
\end{proof}

\begin{example}
(Stochastic volatility models) Consider the stochastic volatility process $(X_{n})$ given by the equation
$$ X_{n}=\sigma_{n} Z_{n}, \qquad n \in \mathbb{Z},$$
where the noise sequence $(Z_{n})$ consists of nonnegative i.i.d.~regularly varying random variables with index $\alpha >0$, and $(\log \sigma_{n})$ is a Gaussian causal ARMA process which is independent of the sequence $(Z_{n})$.

Then $(X_{n})$ satisfies the strong mixing condition with geometric rate (see Davis and Mikosch~\cite{DaMi09}). By virtue of Breiman's result on regularly varying tail of a product of two independent random variables (cf. Proposition 3 in Breiman~\cite{Br65} and equation (16) in Davis and Mikosch~\cite{DaMi09}), every $X_{n}$ is regularly varying with index $\alpha$. From Theorem 2 in Davis and Mikosch~\cite{DaMi09b} it follows that the extremal index of $(X_{n})$ is equal to $1$.

Hence all conditions in Theorem~\ref{t:functconvergence} are satisfied and we obtain the $J_{1}$ convergence of partial maxima process toward an extremal process in $D[0,1]$.
\end{example}

\section{Necessity of the regular variation condition}\label{s:three}

In the i.i.d.~case the $J_{1}$ convergence of the partial maxima processes $M_{n}(\,\cdot\,)$ to an extremal process implies the regular variation property of $X_{n}$'s (cf. Proposition 7.2 in Resnick~\cite{Re07}). In this section we extend this result to the dependence case when clustering of large values do not occur. This can be viewed as a certain converse of Theorem~\ref{t:functconvergence}, but now we do not have to impose the strong mixing condition on the sequence $(X_{n})$. First we state a simple result on the continuity of the projection to the right endpoint in the $J_{1}$ topology. Since its proof is straightforward we omit it here.

\begin{lemma}\label{l:picont}
The function $\pi \colon D[0,1] \to \mathbb{R}$ defined by $\pi(x)=x(1)$ is continuous with respect to the $J_{1}$ topology on $D[0,1]$.
\end{lemma}

\begin{theorem}\label{t:functconvreverse}
Let $(X_{n})$ be a strictly stationary sequence of nonnegative random variables. Suppose the sequence $(X_{n})$ has extremal index $\theta=1$.
If
 $ M_{n}(\,\cdot\,) \dto \widetilde{M}(\,\cdot\,)$
in $D[0,1]$ endowed with the $J_{1}$ topology, where $\widetilde{M}(\,\cdot\,)$
is an extremal process with exponent
measure $\nu$,
then
$$ n  \PP
    ( a_{n}^{-1} X_{1} \in \cdot\,) \xrightarrow{v} \nu(\,\cdot\,) \qquad \textrm{as} \ n \to \infty.$$
\end{theorem}
\begin{proof}
From the functional $J_{1}$ convergence  $ M_{n}(\,\cdot\,) \dto \widetilde{M}(\,\cdot\,)$, by the continuous mapping theorem and Lemma~\ref{l:picont}, we get
$M_{n}(1) \dto \widetilde{M}(1)$, i.e.
$$ \PP \bigg( \bigvee_{i=1}^{n}\frac{X_{i}}{a_{n}} \leq x \bigg) \dto \PP(\widetilde{M}(1) \leq x) = e^{- \nu(x,\infty)} \qquad \textrm{as} \ n \to \infty,$$
for every $x>0$. Let $(\widehat{X}_{n})$ be the associated independent sequence of $(X_{n})$, i.e. $(\widehat{X}_{n})$ is an i.i.d. sequence with $\widehat{X}_{1} \eqd X_{1}$. Then by Theorem 2.2.1 in Leadbetter and Rootz\'{e}n~\cite{LeRo88}
\begin{equation*}\label{e:LRpowernew}
 \PP \bigg( \bigvee_{i=1}^{n} \frac{\widehat{X}_{i}}{a_{n}} \leq x \bigg) \to e^{-\frac{1}{\theta} \nu(x,\infty)} \qquad \textrm{as} \ n \to \infty.
\end{equation*}
From this, taking into account the equivalence of the regular variation property and the weak convergence of maxima for an i.i.d.~sequence (cf. Lemma 1.2.2 in Leadbetter and Rootz\'{e}n~\cite{LeRo88} and Proposition 7.1 in Resnick~\cite{Re07}) and the fact that $\theta=1$, we obtain
$n \PP (a_{n}^{-1}\widehat{X}_{1} > x) \to \nu(x,\infty)$.
This implies
$$ n  \PP
    ( a_{n}^{-1} X_{1} \in \cdot\,) \xrightarrow{v} \nu(\,\cdot\,) \qquad \textrm{as} \ n \to \infty$$
    (cf. Lemma 6.1 in Resnick~\cite{Re07}).\qed
\end{proof}

When $\theta < 1$, i.e. clustering of large values occurs, then generally we can not have the $J_{1}$ convergence of the partial maxima process (see Example 5.1 in Krizmani\'{c}~\cite{Kr14}), but convergence in the weaker $M_{1}$ topology may still hold (cf. Theorem 4.1 in Krizmani\'{c}~\cite{Kr14}). And if it holds then we can recover the regular variation property, as is shown in the next result, which generalizes Theorem~\ref{t:functconvreverse}.

\begin{theorem}
Let $(X_{n})$ be a strictly stationary sequence of nonnegative random variables. Suppose the sequence $(X_{n})$ has extremal index $\theta \in (0,1]$.
If
 $ M_{n}(\,\cdot\,) \dto \widetilde{M}(\,\cdot\,)$
in $D[0,1]$ endowed with the $M_{1}$ topology, where $\widetilde{M}(\,\cdot\,)$
is an extremal process with exponent
measure $\nu$,
then
$$ n  \PP
    ( a_{n}^{-1} X_{1} \in \cdot\,) \xrightarrow{v} \frac{1}{\theta}\nu(\,\cdot\,) \qquad \textrm{as} \ n \to \infty.$$
\end{theorem}
\begin{proof}
The proof is practically the same as the proof of Theorem~\ref{t:functconvreverse}, with the difference that instead of the Lemma~\ref{l:picont} one has to use the corresponding result for the continuity of the function $\pi$ with respect to the $M_{1}$ topology on $D[0,1]$ (see Theorem 12.5.1 (iv) in Whitt~\cite{Whitt02}).\qed
\end{proof}

\begin{remark}
In the light of the results presented in this article one can raise a question whether the property that the extremal index equals to $1$ is a necessary condition for the $J_{1}$ convergence of the partial maxima processes. The answer is negative. Consider the finite order moving maxima defined by
$$ X_{n} = \max \{\xi_{n}, \xi_{n-1}\}, \qquad n \in \mathbb{Z},$$
 where $\xi_{i}$, $i \in \mathbb{Z}$, are i.i.d.
  unit Fr\'{e}chet random variables, i.e. $\PP (\xi_{i} \leq x) = e^{-1/x}$ for $x>0$. Then $(X_{n})$ is strongly mixing and jointly regularly varying, $\theta = 1/2$ and the corresponding partial maxima processes $M_{n}(\,\cdot\,)$ converge to an extremal process in the $M_{1}$ topology (see Example 5.1 in Krizmani\'{c}~\cite{Kr14}). This implies the convergence of the finite dimensional distributions of $M_{n}(\,\cdot\,)$. Tightness with respect to $J_{1}$ topology can be obtained in a standard manner (and thus we omit it here), and we conclude that partial maxima processes converge also in the $J_{1}$ topology. A "big value" $\xi_{i}$ produces two successive big values $X_{i}$ and $X_{i+1}$ in the sequence $(X_{n})$, but with the same magnitude, and this produces only one jump in the maxima process $M_{n}(\,\cdot\,)$ (at $t=i/n$), thus allowing the $J_{1}$ convergence to hold.

  \end{remark}

\section*{Acknowledgements}
This work has been supported in part by Croatian Science Foundation under the project 3526.

\end{document}